\documentclass[11pt]{article}

\usepackage{latexsym,color,amsmath,amsthm,amssymb,amscd,amsfonts}
\usepackage{graphicx}
\usepackage{hyperref}
\usepackage{cancel}

\newtheorem{conj}{Conjecture}[section]
\newtheorem{theorem}[conj]{Theorem}

\newtheorem{remark}[conj]{Remark}
\newtheorem{lemma}[conj]{Lemma}
\newtheorem{proposition}[conj]{Proposition}

\newtheorem{definition}[conj]{Definition}

\newtheorem{ex}[conj]{Examples}

\newcommand\independent{\protect\mathpalette{\protect\independent}{\perp}} 
\def\independent#1#2{\mathrel{\rlap{$#1#2$}\mkern2mu{#1#2}}}

\newcommand{\R}{\mathbb{R}}

\renewcommand{\P}{\mathbb{P}}

\newcommand{\E}{\mathbb{E}}

\DeclareMathOperator{\Var}{Var}

\newcommand{\Z}{\mathcal{Z}}

\renewcommand{\Z}{\mathbb{Z}}

\newcommand{\N}{\mathbb{N}}

\date{\vspace{-5ex}}

\author{Abdulmajeed Alqasem\thanks{Department of Mathematics, King Saud University, Riyadh 11451, Saudi Arabia, {\it aalqasem1@ksu.edu.sa}} 
, Heshan Aravinda\thanks{ Department of Mathematics and Statistics, University of Wyoming, 1000 E. University Ave, Laramie, WY 82071, USA, {\it hpathira@uwyo.edu}}, Arnaud Marsiglietti\thanks{ Department of Mathematics, University of Florida, Gainesville, FL 32611, USA, {\it a.marsiglietti@ufl.edu}}}

\title{Majorization, Entropy and Concentration Inequalities for Discrete $\alpha$-Concave Random Variables}

\begin{document}

\maketitle

\begin{abstract}\noindent
We establish a maximum-variance inequality for discrete $\alpha$-concave random variables in the regime $-1/3<\alpha<0$, extending corresponding results beyond the discrete log-concave setting. As applications, we obtain a reverse entropy power inequality and bounds for the L\'evy concentration function. We further establish polynomial-type concentration inequalities, reflecting the heavy-tailed nature of the class. A main tool in our approach is a convex majorization principle that reduces the relevant extremal problems to suitable $\alpha$-affine distributions.

\end{abstract}


\section{Introduction}

Let $\alpha \in [-\infty, \infty]$. For real numbers $a,b \geq 0$ and $\theta \in [0,1]$, define the {\it generalized means of order $\alpha$} as follows,
\[
M_\alpha^\theta(a,b):=
\begin{cases}
\big((1-\theta)a^\alpha+\theta b^\alpha\big)^{1/\alpha}, & \alpha \notin \{-\infty, 0 , \infty\}\\
\min\{a,b\}, &\alpha=-\infty\\
a^{1-\theta}b^\theta, & \alpha=0 \\
\max\{a,b\}, & \alpha = \infty
\end{cases}.
\]
These generalized means give rise to a hierarchy of concavity properties for probability measures. In particular, a probability measure $\mu$ on $\mathbb{R}^n$ is said to be $\alpha$-concave, $\alpha \in [-\infty,+\infty]$, if for every $\theta \in [0,1]$ and for all compact sets $A,B \subset \R^n$ such that $\mu(A)\mu(B)>0$, one has
\begin{equation*}
\mu((1-\theta)A + \theta B) \geq  M_\alpha^{\theta}(\mu(A),\mu(B)).
\end{equation*}
Here, the parameter $\alpha$ is understood as a convexity parameter. It is easy to see by Jensen's inequality that any $\alpha$-concave measure is $\gamma$-concave for all $\gamma \leq \alpha$. In particular, any $\alpha$-concave measure is $-\infty$-concave. The class of $-\infty$-concave measures is called convex measures, and the class of $0$-concave measures is called log-concave measures.

Likewise, a function $f \colon \R^n \to [0,+\infty)$ is {$\kappa$-concave}, $\kappa \in [-\infty,+\infty]$, if for all $\theta \in [0,1]$, for all $x,y \in \R^n$ such that $f(x)f(y)>0$, one has
\begin{equation*}
f((1-\theta)x + \theta y) \geq M_{\kappa}^{\theta}(f(x),f(y)).
\end{equation*}

The relationship between $\alpha$-concave probability measures and $\kappa$-concave functions is described in the celebrated work of Borell \cite{B1,B2}. Namely, an $\alpha$-concave probability measure on $\mathbb{R}^n$ that is absolutely continuous (with respect to Lebesgue measure) has a $\kappa$-concave probability density function, where $\kappa = \frac{\alpha}{1-\alpha n}$. The class of $\alpha$-concave distributions provides a natural extension of the extensively studied class of log-concave distributions. For negative values of the convexity parameter, it also includes heavy-tailed distributions that fall outside the log-concave setting, such as Cauchy-type distributions. Therefore $\alpha$-concavity provides a framework that interpolates between the well-developed log-concave theory and broader classes of heavy-tailed probability measures. We refer to \cite{B1,B2} for further background on $\alpha$-concave measures.

Motivated by the ``continuous'' theory, we consider the following natural discrete counterpart of $\alpha$-concavity on the set of integers $\mathbb{Z}$.

\begin{definition}

    A sequence of non-negative real numbers $\{x_k\}_{k \in \Z}$ is $\alpha$-concave if for all $k \in \Z$,
\begin{equation}\label{def-alpha}
x_k \geq M_\alpha^{1/2}(x_{k-1},x_{k+1}),
\end{equation}
and the sequence $\{x_k\}$ has no internal zeros in the sense that $x_m x_n > 0$ implies that $x_k > 0$ for $m \leq k \leq n$.

\end{definition}

A sequence is called $\alpha$-affine if equality holds in \eqref{def-alpha} throughout its support. As in the continuous setting, the limiting cases $\alpha \in \{-\infty, 0, \infty\}$ are interpreted by continuity. For example, $0$-concavity coincides with (discrete) log-concavity.

An integer-valued random variable $X$ is called discrete $\alpha$-concave if its probability mass function forms an $\alpha$-concave sequence.

\begin{ex}\label{examples}

Several families of discrete distributions fall within the class of discrete $\alpha$-concave distributions. We record a few examples and briefly verify condition \eqref{def-alpha}. 
\end{ex}

\begin{enumerate}
    \item {\bf Zeta (discrete Pareto) distribution} with the parameter $s>1$ is $(-1/s)$-affine. Indeed, its probability mass function is given by 
    $$ f(k) = \frac{k^{-s}}{\zeta(s)}, \quad k=1,2,\dots, $$ where $\zeta(s)$ is the Riemann-zeta function. Since $(\zeta(s)f(k))^{-1/s} = k$ is affine in $k$, the equality holds in \eqref{def-alpha}.
        
    \item {\bf The Zipf distribution} with the parameter $s>0$ on $N$ elements is $(-1/s)$-affine. The probability mass function of the Zipf distribution is given by 
    $$ f(k) = \frac{k^{-s}}{H_{N,s}}, \quad k=1,2,3,\dots,N\,, $$
    where $H_{N,s} = \sum_{k=1}^N k^{-s}$. Thus $(H_{N,s}f(k))^{-1/s} = k$, which is affine in $k$. \vskip2mm
    
    \item {\bf The Yule-Simon distribution} with an integer shape parameter $\rho \in \{1,2, \dots\}$ is $\alpha$-concave with
    \begin{equation}\label{alpha}
    \alpha = - \sum_{i=1}^{\rho+1} \frac{1}{i^2} / \left( \sum_{i=1}^{\rho + 1} \frac{1}{i} \right)^2.
    \end{equation}
    Recall that the probability mass function of the Yule-Simon distribution with shape parameter $\rho > 0$ is given by 
    $$ f(k, \rho) = \rho B(k,\rho+1), \quad k\geq 1, $$
    where $B$ is the beta function. When $\rho$ is an integer, one may write
    $$ f(k,\rho) = \frac{\rho \rho!}{k(k+1)\cdots(k+\rho)}, $$
    and the value of $\alpha$ in \eqref{alpha} is deduced from the second derivative of the function $f(x,\rho)^{-\alpha}$, $x \in \R$, evaluated at $x=1$.

    \end{enumerate} 

These distributions arise naturally in a variety of applications. The Zeta distribution, for example, has been used to give probabilistic proofs of certain number-theoretic identities \cite{K, LH, N, T}. Zipf distributions are widely used to model power-law behavior in areas such as linguistics, information theory, and network data \cite{L}, while the Yule-Simon distribution arises naturally in preferential attachment models \cite{PPS, S}.

Compared with the discrete log-concave setting, considerably less is known about probabilistic inequalities and properties for discrete $\alpha$-concave distributions when $\alpha<0$. The goal of this paper is to develop several probabilistic inequalities for discrete $\alpha$-concave distributions, with particular emphasis on the range $-1/3<\alpha<0$, which guarantees existence of second moment. Our approach is based on two complementary ideas. First, we construct a continuous $\alpha$-concave extension of a discrete $\alpha$-concave sequence, allowing certain estimates from the continuous setting to be transferred to $\mathbb{Z}$. Second, we extend a well-known majorization principle that reduces several extremal problems to suitable $\alpha$-affine distributions. These tools yield, in particular, the following maximum-variance inequality, which is the first main result in the paper. Related inequalities connecting the maximum probability, entropy, and variance have previously been studied in the discrete log-concave setting; see, for example \cite{Ara, BMM, GMPP, JMNS}.

\begin{theorem}\label{maximum-vs-variance}

Let $\alpha \in (-1/3 , 0)$. If $X$ is a discrete $\alpha$-concave random variable on $\Z$ with the maximum probability $M(X)$, then
    \begin{equation*}
    M(X)^2 \left( 1 + \dfrac{1}{2} c(\alpha) \Var(X) \right) \leq 1,\end{equation*}
where $c(\alpha) = \dfrac{ (1+ 2\alpha)(1+ 3\alpha)}{(1+\alpha)^2}$.

\end{theorem}
We further apply Theorem \ref{maximum-vs-variance} to obtain bounds for the L\'evy concentration function and for discrete R\'enyi entropy. The latter leads to a reverse entropy power inequality (EPI) for sums of independent discrete $\alpha$-concave random variables described in Theorem \ref{Reverse_EPI}. We refer to \cite{HAT, HV, JY, MMR, MM26, WM} for further investigation on discrete entropy power inequalities.

For an integer-valued random variable $X$ with probability mass function $f$, recall that the R\'{e}nyi entropy of order $r>0$, $r \neq 1$, is defined by,  
$$ H_{r}(X) =\dfrac{1}{1-r}\log \left( \sum_{k \in \mathbb{Z}}\,f(k)^{r} \right). $$
The R\'{e}nyi entropy power of order $r$ is,
$$ N_{r}(X) = e^{2H_{r}(X)} = \left( \displaystyle \sum_{k \in \mathbb{Z}} f(k)^{r}\right)^{\frac{2}{1-r}}. $$
Since $N_r(X)\geq 1$ for an arbitrary discrete random variable $X$, the quantity $\Delta_r(X) = N_r(X)-1$ is often considered to play the role of the continuous R\'{e}nyi entropy power as it reflects the size of the variance more accurately. Our next main result is the following form of reverse EPI.

\begin{theorem}\label{Reverse_EPI}

Let $\alpha \in (-1/3 , 0)$. Let $X_1, \dots, X_n$, $n \geq 1$, be independent discrete $\alpha$-concave random variables. Then, for all $r>1$,
$$ \Delta_r \left( \sum_{i=1}^n X_i \right) \leq C(r,\alpha) \sum_{i=1}^n \Delta_r(X_i), $$
with $C(r,\alpha) = \dfrac{8(3r-1)}{(r-1)c(\alpha)}$, where $c(\alpha) = \dfrac{ ( 1+ 2\alpha)(1 + 3\alpha)}{(1+\alpha)^2}$.
    
\end{theorem}
Finally, we establish a polynomial-type concentration inequality. Since exponential moments need not be finite in the $\alpha$-concave setting, $\alpha < 0$, the usual Chernoff approach (see, e.g., \cite{BLM, V}) is not directly available. Instead, we utilize the single-crossing comparison with $\alpha$-affine distributions underlying the majorization argument of Section \ref{majorization-sec}.

\begin{theorem} \label{concentration_bound}

Let $\alpha \in (-1 , 0)$. If $X$ is a discrete $\alpha$-concave random variable with a mode $m$, then for all $t>0$,
\begin{equation} \label{concentration_alpha_concave}
     \mathbb{P}(|X-m|\geq t) \leq \left(1 + \dfrac{M(X)\,t}{-(1+ 1/\alpha)}\right)^{1+ 1/\alpha} + M(X)\left(1+ \dfrac{t}{-(1+1/\alpha)}\right)^{1+ 1/\alpha}.
\end{equation}

\end{theorem}

Note that since $\mathbb{P}(|X-m|\geq t)\leq 1-M(X)$ for $t>0$, the right-hand side of \eqref{concentration_alpha_concave} may be replaced with its minimum with $1-M(X)$.

\begin{remark}\label{rem-concentration}

The polynomial nature of the bound in Theorem \ref{concentration_bound} is consistent with the heavy-tailed behavior of $\alpha$-concave distributions for $\alpha<0$. Moreover, as $\alpha \to 0^-$, the right-hand side of \eqref{concentration_alpha_concave} approaches 
$$e^{-M(X)\,t}+ M(X)e^{-t}\,,$$
exhibiting the expected transition to exponential-type concentration in the log-concave limit.
\end{remark}

\noindent
{\bf Outline of the Paper.} Section \ref{majorization-sec} establishes a Whitt-type majorization principle. In Section \ref{continuous-discrete}, we develop a continuous extension of discrete $\alpha$-concave sequences and derive some moment estimates.  In Section \ref{max-var}, we apply the majorization principle directly to the discrete setting to prove Theorem \ref{maximum-vs-variance}. Sections \ref{Levy-concentration} and \ref{Reverse-EPI} give applications to the L\'evy concentration function and an entropy power inequality, respectively, including the proof of Theorem \ref{Reverse_EPI}. Finally, Section \ref{concentration} establishes the concentration inequality of Theorem \ref{concentration_bound}.

\section{Convex majorization}\label{majorization-sec}

In this section, we introduce a Whitt-type convex majorization principle that will be used in Sections \ref{continuous-discrete} and \ref{max-var}. The following theorem generalizes the majorization argument of \cite[Theorem 1]{MM22} by replacing equality of the first moment with an inequality. As a consequence, the conclusion is restricted to non-decreasing convex functions rather than arbitrary convex functions. This restriction is sufficient for our applications to moment bounds on the positive real line. We refer to \cite[Section 2]{MM22} for further details on convex majorization (see also \cite{KN, W}). 

 \begin{theorem}\label{majorization}
    Let $\mu$ be a measure on $[0, +\infty)$. Suppose that $f$ and $g$ are Borel measurable functions on $[0, +\infty)$ such that
    \begin{align*}
        \int_0^\infty f(x) d\mu(x) = \int_0^\infty g(x) d\mu(x) 
    \end{align*}
    and 
    \begin{align*}
        \int_0^\infty x f(x) d\mu(x) \leq \int_0^\infty x g(x) d\mu(x)
    \end{align*}
    and the integrals are finite. If there exists an interval $I := [a,b] \subseteq [0,\infty)$ such that $g(y) \leq f(y)$ for $y \in I$ while $g(y) \geq f(y)$ for $y \notin I$, then $\varphi \colon (0,\infty) \to \mathbb{R}$ convex nondecreasing implies,
    \begin{align*}
        \int_0^\infty \varphi(x) f(x) d\mu(x) \leq \int_0^\infty \varphi(x) g(x) d\mu(x).
    \end{align*}
\end{theorem}

\begin{proof}
    We adapt the arguments from \cite[Proof of Theorem 1]{MM22}. Taking $\tilde{\varphi}(x) = \varphi(x) - (mx+k)$ with
    $$ m := \frac{\varphi(b)- \varphi(a)}{b-a}, \qquad k := \frac{\varphi(a)b - \varphi(b)a}{b-a}, $$
    yields $\tilde{\varphi}(a) = \tilde{\varphi}(b) = 0$. Note that $m \geq 0$ as $\varphi$ is nondecreasing, therefore
    \[
        \int_0^\infty (g -f)(x) (\varphi - \tilde{\varphi})(x) d\mu(x) = \int_0^\infty (g - f)(x) (mx + k) d\mu(x) = m \int_0^\infty x(g-f)(x) d\mu(x) \geq 0.
    \]
    Hence,
    $$ \int_0^\infty (g - f)\varphi d\mu \geq \int_0^\infty (g - f) \tilde{\varphi} d\mu \geq 0, $$
    where the last inequality follows since $(g-f) \tilde{\varphi}$ is a non-negative function, by assumptions on $f$ and $g$, and by construction of $\tilde{\varphi}$.
\end{proof}

\section{From continuous to discrete $\alpha$-concavity}\label{continuous-discrete}

Let us begin with a construction that associates an $\alpha$-concave function on $\mathbb{R}$ with every discrete $\alpha$-concave sequence. This provides a convenient way to transfer certain estimates from the continuous setting to the discrete one. The construction is a direct adaptation of \cite{BMM} and will be referred to as the $\alpha$-piecewise linear extension.

%We start this section by establishing an equivalence between $\alpha$-concave functions and $\alpha$-concave sequences.

\begin{proposition} \label{alpha_piecewise_construction}
    A sequence $\{x_k\}$ is $\alpha$-concave if and only if there exists an $\alpha$-concave function $f$ such that $f(k) = x_k$ for all $k \in \Z$.
\end{proposition}

\begin{proof}
Let $\Delta$ be the smallest closed interval
containing the support of $\{x_k\}$. Define $V(k) = x_k^{\alpha}$, $k \in \Delta \cap \mathbb{Z}$. Let us extend $V$ linearly on every segment $[k,k+1] \subset \Delta$, and put $V = \infty$ 
outside $\Delta$ (when $\Delta \neq \mathbb{R}$). We claim that $V$ is finite 
and convex on $\Delta$. Indeed, whenever $[k-1,k+1] \subset \Delta$, inequality \eqref{def-alpha} is equivalent to
$$
V(k) - V(k-1) \leq V(k+1) - V(k).
$$
But $V(k) - V(k-1) = V'(x)$ on $(k-1,k)$, which implies that for all $x \leq y$ such that $x, y \notin \Z$, $V'(x) \leq  V'(y)$. This property together with the fact that $V$ is continuous on $\{V < \infty\}$ allows us to conclude that $V$ is convex. As a result, we obtain an $\alpha$-concave function
$$
f(x) = V^{\frac{1}{\alpha}}(x), \quad x \in \R,
$$
which coincides with $x_k$ on all integers $k$.\vskip2mm \noindent
Conversely, if $f$ is an $\alpha$-concave function such that $f(k) = x_k$ for all $k\in \mathbb{Z}$, then applying $\alpha$-concavity to the midpoint $k = \frac{1}{2}(k-1) + \frac{1}{2}(k+1)$ gives inequality \eqref{def-alpha} whenever $x_k, x_{k-1}$ and $x_{k+1}$ are in the support of $\{x_k\}$.
\end{proof}
We now apply this construction to the probabilistic setting. Let $X$ be a discrete $\alpha$-concave random variable with probability mass function $f(k)$, $k\in \mathbb{Z}$, and let $f(x)$, $x\in\mathbb{R}$, denote its $\alpha$-piecewise linear extension constructed in Proposition \ref{alpha_piecewise_construction}. Although $f(x)$ need not be a probability density function, the following proposition shows that its integral can be compared with the total mass of the probability mass function. Recall that $\Delta$ denotes the smallest closed interval containing the support of $\{f(k)\}$.

\begin{proposition}\label{two_sided_bound_mode}

Let $X$ be a discrete $\alpha$-concave random variable, and let $m$ be a mode of $X$. If $|\Delta| \geq 2$, then
$$
1 - f(m) \leq \int_{-\infty}^\infty f(x)\,dx \leq 1.
$$
    
\end{proposition}

\begin{proof}
Let $m \in \Z$ be a mode of $X$. Then $\{f(k)\}$ is non-increasing for $k \geq m$ and is non-decreasing for $k \leq m$. Hence
$$
\int_k^{k+1} f(x)\,dx \, \geq \, f(k+1) \quad {\rm for} \ k \geq m, \qquad
\int_{k-1}^k f(x)\,dx \, \geq \, f(k-1) \quad {\rm for} \ k \leq m.
$$
Summing over all $k$, we get
$$
\int_{-\infty}^\infty f(x)\,dx \geq 1 - f(m).
$$
Also, since $V$ is linear on each interval $[k,k+1] \subset \Delta$, 
$f(x)$ is convex, and therefore its graph lies below the segment
connecting the points $(k,f(k))$ and $(k+1,f(k+1))$. Hence
$$
\int_k^{k+1} f(x)\,dx \leq \frac{f(k) + f(k+1)}{2}.
$$
Summing over all $k$, we obtain the desired result.
\end{proof}

Similar bounds can be derived for higher-order moments of arbitrary discrete unimodal random variables. For our purpose, we state the result only for discrete $\alpha$-concave random variables.

\begin{proposition}\label{moments}
Let $X$ be a discrete $\alpha$-concave random variable with density $f(k)$ with mode $m$. If $f(m) \leq 1/2$, then the $\alpha$-piecewise linear extension $f(x)$ satisfies
\begin{equation}\label{moments-eq}
   \frac{\E[|X-m|^p]}{p+1} \leq \int_{\R} |x-m|^p f(x)\,dx \, \leq \frac{2^{p+1}}{p+1} \E[|X-m|^p],
\end{equation}
the lower bound being also true when $f(m) \geq 1/2$.
\end{proposition}
\begin{proof}
%Since $f$ is unimodal, it has a mode at $m$. In particular, $f(k) \leq f(k+1)$ for $k \leq m-1$, and $f(k+1) \leq f(k)$ for $k \geq m$.
By translation, assume that $m = 0$. Note that for $k \geq 0$,
$$ \int_k^{k+1} |x|^p f(x) dx \leq f(k) \int_k^{k+1} |x|^p dx = f(k) \frac{(k+1)^{p+1} - k^{p+1}}{p+1}, $$
and for $k \leq 0$,
$$ \int_{k-1}^{k} |x|^p f(x) dx \leq f(k) \int_{k-1}^{k} |x|^p dx = f(k) \frac{(|k|+1)^{p+1} - |k|^{p+1}}{p+1}. $$
Since for $k \geq 1$,
$$ (k+1)^{p+1} - k^{p+1} = \sum_{j=0}^{p+1} \binom{p+1}{j} k^j - k^{p+1} = \sum_{j=0}^{p} \binom{p+1}{j} k^j \leq k^p \sum_{j=0}^{p} \binom{p+1}{j} = k^p (2^{p+1} - 1), $$
after summing over all $k$ we deduce from above that
$$ \int_{\R} |x|^p f(x) dx \leq \frac{1}{p+1} \left( f(0) + \sum_{k \neq 0} |k|^p f(k) (2^{p+1} - 1) \right) = \frac{1}{p+1} \left( f(0) + (2^{p+1} - 1) \E[|X|^p] \right). $$
If $f(0) \leq \frac{1}{2}$, then necessarily 
$\E[|X|^p] \geq \P\{|X| \geq 1\} \geq \frac{1}{2} \geq f(0)$, so that
$$
\int_{\R} |x|^p f(x)\,dx \, \leq \frac{2^{p+1}}{p+1} \E[|X|^p].
$$
The lower bound is more interesting for application in this article. Note that for $k \geq 0$,
$$ (k+1)^{p+1} - k^{p+1} = \sum_{j=0}^{p} \binom{p+1}{j} k^j \geq  \sum_{j=0}^{p} \binom{p}{j} k^j = (k+1)^p. $$
Therefore, for $k \geq 0$,
$$ \int_k^{k+1} |x|^p f(x) dx \geq f(k+1) \int_k^{k+1} |x|^p dx = f(k+1) \frac{(k+1)^{p+1} - k^{p+1}}{p+1} \geq f(k+1) \frac{(k+1)^p}{p+1}. $$
For $k \leq 0$, 
$$ \int_{k-1}^k |x|^p f(x) dx \geq f(k-1) \int_{|k|}^{|k|+1} x^p dx = f(k-1) \frac{(|k|+1)^{p+1} - |k|^{p+1}}{p+1} \geq f(k-1) \frac{(|k|+1)^p}{p+1}. $$
Summing over all $k$ leads to
$$ \int_{\R} |x|^p f(x) dx \geq \sum_{k \geq 0} f(k+1) \frac{(k+1)^p}{p+1} + \sum_{k \leq 0} f(k-1) \frac{(|k|+1)^p}{p+1} = \frac{\E[|X|^p]}{p+1}. $$
Translating back to an arbitrary mode $m$, we obtain
$$ \int_{\R} |x-m|^p f(x) dx \geq \frac{\E[|X-m|^p]}{p+1}. $$
\end{proof}

\begin{remark}\label{rem-moments}
    Contrary to the log-concave case, not all moments of $\alpha$-concave probability density functions or of discrete $\alpha$-concave random variables, are necessarily finite when $\alpha<0$. In both settings, moments are finite up to an order determined by the corresponding concavity parameter. We first recall the argument in the continuous setting (see \cite{Bobkov, FGP}) and include the details for completeness.
\end{remark}

 Let $s>0$, and let $\Bar{X}$ be a random variable with a $(-1/s)$-concave density $\bar{f}$. Then there exists a convex function $V \colon \mathbb{R} \to [0, \infty]$ such that $\bar{f}=V^{-s}$. By the integrability of $\bar{f}$, necessarily $V(x) \to \infty$ as $|x| \to \infty$. Therefore, by convexity of $V$, there exists $c>0$ such that $V(x) \geq c (1+|x|)$. Equivalently, $\bar{f}(x) \leq c (1+|x|)^{-s}$. Thus, for $p>0$, 
$$ \mathbb{E}[|\Bar{X}|^p] = \int_{-\infty}^\infty |x|^p \bar{f}(x) \,dx \leq c\, \int_{-\infty}^\infty |x|^p (1+|x|)^{-s} \,dx = 2c \int_0^\infty x^p (1 + x)^{-s}\,dx. $$ 
The last integral can be rearranged as $\int_0^\infty x^{a-1} (1+x)^{-a-b}\, dx$, where $a = p+1$ and $b = s - p-1$. But this represents a Beta function $B (a,b)$, which converges if $a,b>0$. Therefore, $\mathbb{E}[|\Bar{X}|^p]< \infty$ if $p< s-1$. 

\vskip2mm
\noindent
Returning to the discrete setting, let $X$ be a discrete $\alpha$-concave random variable with $\alpha<0$. By Proposition \ref{alpha_piecewise_construction} its $\alpha$-piecewise linear extension is also $\alpha$-concave. Therefore, its $p$-th absolute moment is finite whenever $\alpha > -1/(p+1)$. Proposition \ref{moments} then implies that $\mathbb{E}[|X|^p]< \infty$ for the same $\alpha$ threshold.

\vskip2mm
\noindent
Note that this threshold is consistent with the fact that log-concave random variables have finite moments of all orders: as $\alpha\to0^-$, the threshold $p^\star = -1/\alpha - 1 \to \infty$.

\vskip2mm
In analogy with the log-concave case, Propositions \ref{two_sided_bound_mode}-\ref{moments} allow us to extend the second-moment and variance inequalities obtained in \cite{BMM} to the class of discrete $\alpha$-concave random variables.

\begin{proposition}\label{continuous_max_second_moment}

Let $\alpha \in (-1/3,0)$. If $\bar{X}$ is a continuous random variable with an $\alpha$-concave density with mode at the point $m$, then
\begin{equation*}
    \dfrac{c(\alpha)}{2} M^2(\bar{X})\,\mathbb{E}[(\bar{X}-m)^2] \leq 1\,,
\end{equation*}
    where $c(\alpha) = \dfrac{(1+2 \alpha)(1+3\alpha)}{{(1+\alpha)^2}}$.
    
\end{proposition}

As an immediate consequence, we have the following.

\begin{proposition} \label{sub_optimial_max-variance}

    Let $\alpha \in (-1/3,0)$. If $X$ is discrete $\alpha$-concave with mode at $m$, then 
    \begin{equation} \label{sub_optimial_bound}
    \dfrac{c(\alpha)}{6} M^2(X)\, \mathbb{E}[(X-m)^2] \leq 1\,,\end{equation}
    where $c(\alpha) = \dfrac{(1+2 \alpha)(1+3\alpha)}{{(1+\alpha)^2}}$. In particular,
    $$\dfrac{c(\alpha)}{6} M^2(X)\, \Var(X) \leq 1\,.$$
    
    \end{proposition}
    
Note that as $\alpha \to 0^{-}$, $c(\alpha)$ approaches $1$, recovering \cite[Proposition 6.2]{BMM} in the log-concave setting.

\begin{proof}[Proof of Proposition \ref{sub_optimial_max-variance}]
    Denote by $f$ the probability mass function of $X$. Note that its $\alpha$-piecewise linear extension $f(x)$ is $\alpha$-concave. Let $\bar{X}$ be a continuous random variable with density $\bar{f}(x) = \dfrac{1}{C}f(x)$, where $C = \int_{\mathbb{R}}f(x)\, dx$. It follows from Proposition \ref{two_sided_bound_mode} that $C \leq 1$. Moreover, $$M(\Bar{X}) = \dfrac{f(m)}{C}\,\,\, \text{and}\,\,\, \mathbb{E}[(\Bar{X}-m)^2] = \dfrac{1}{C} \displaystyle \int_{\mathbb{R}}(x-m)^2f(x)\,dx.$$ Hence, by Proposition \ref{continuous_max_second_moment} applied to $\bar{f}$, we get
    $$ \dfrac{c(\alpha)}{2} f(m)^2 \int_{\mathbb{R}} (x-m)^2 f(x)\,dx \leq 1. $$
  We conclude since $\dfrac{\mathbb{E}[(X-m)^2]}{3} \leq \int_{\mathbb{R}}(x-m)^2 f(x)\,dx$, due to Proposition \ref{moments}.
\end{proof}

Let us now turn to the proof of Proposition \ref{continuous_max_second_moment}, which is based on majorization in the convex order.
%The idea is to adapt the single crossing argument used in the proof of \cite[Proposition 6.1]{BMM} to our setting.

\begin{proof}[Proof of Proposition \ref{continuous_max_second_moment}] Let us denote by $\bar{f}$ the probability density function of $\bar{X}$. By translation and scaling invariance, we may assume that $m=0$ and $\bar{f}(0) = 1$. Writing $\bar{f} = V^{1/\alpha}$ for some convex function $V$ on $\mathbb{R}$, we then have $V(0) =  1$. Let $a = \P(\Bar{X}<0)$ and $b = \P(\Bar{X}>0)$, and assume without loss of generality that $a,b > 0$. Define the function $\bar{V}$ such that
\[\bar{V}(x) = \begin{cases} 
    1 +u x, &  \mbox{ if } x \geq 0\\       
     1 - vx,  & \mbox{ if } x<0
     \end{cases}
\] 
and let $\bar{g} = \bar{V}^{1/\alpha}$, where $u,v >0$ are such that $\displaystyle \int_{-\infty}^0 \bar{g} = a$ and $\displaystyle \int_0^{\infty} \bar{g} = b$. Note that by construction, $\bar{g}$ is also a probability density function and its restriction to $[0,+\infty)$ is $\alpha$-affine. Since on $[0,+\infty)$ the function $V$ is convex, $\bar{V}$ is affine and $V(0)=\bar{V}(0)$, then the sign of $V-\Bar{V}$ can change from negative to positive at most once. Moreover, since $\displaystyle \int_0^\infty \bar{f}(x)\,dx = \displaystyle \int_0^\infty \Bar{g}(x)\,dx $, this sign change must occur unless $V = \bar{V}$. Thus for some $x_0>0$, $V(x) \leq \bar{V}(x)$ for $0\leq x<x_0$ and $V(x) \geq \bar{V}(x)$ for $x>x_0$. Equivalently, $\bar{f} \geq \bar{g}$ for $x<x_0$ and $\bar{f} \leq \bar{g}$ for $x>x_0$. Denoting by $\bar{Z}$ a random variable with density $\bar{g}$, the above one-crossing property necessarily implies that for all $t \geq 0$, 
$$ \P(\bar{Z} \geq t) \geq \P(\bar{X} \geq t). $$
As a consequence, by the layer-cake representation,
$$ \int_0^\infty x \bar{f}(x) d\mu(x) \leq \int_0^\infty x \bar{g}(x) d\mu(x), $$
where $\mu$ is the Lebesgue measure. Applying Theorem \ref{majorization} with $\mu$ as the Lebesgue measure yields 
$$ \int_0^\infty \varphi(x) \bar{f}(x) d\mu(x) \leq \int_0^\infty \varphi(x) \bar{g}(x) d\mu(x), $$
for any nondecreasing convex function $\varphi$. By taking $\varphi(x) = x^2$, which is nondecreasing on $[0, +\infty)$, we have
$$  \int_0^{\infty} x^2 \bar{f}(x) dx \leq \int_0^{\infty} x^2 \bar{g}(x) dx. $$
Applying the same argument on $(-\infty, 0]$ gives $\int_{-\infty}^0 x^2 \bar{f}(x) \,dx \leq \int_{-\infty}^0 x^2 \bar{g}(x) \,dx $. Therefore,
\begin{equation*}
\E[\bar{X}^2] = \int_{-\infty}^\infty x^2 \bar{f}(x) \, dx \leq \int_{-\infty}^\infty x^2 \bar{g}(x)\,dx.
\end{equation*}
It remains to evaluate the last integral. First, we find constants $u$ and $v$ in terms of $a,b$ and $\alpha$.
$$ \int_{-\infty}^0 \bar{g}(x)\,dx = \int_{-\infty}^0 (1-vx)^{1/\alpha}  = -\dfrac{\alpha}{v(\alpha+1)} = a \implies v = -\dfrac{\alpha}{a(\alpha+1)}. $$
Similarly, it can be shown that, $u =- \dfrac{\alpha}{b(\alpha+1)}$. Finally, 
 $$ \int_{-\infty}^\infty x^2 \bar{g}(x)\,dx =  \int_{-\infty}^0 x^2 \left(1+ \dfrac{\alpha}{a(\alpha+1)}x\right)^{1/\alpha}\,dx+  \int_0^\infty x^2 \left(1- \dfrac{\alpha}{b(\alpha+1)}x\right)^{1/\alpha}\,dx, $$
 which can be shown to be equal to $\dfrac{2(1 +\alpha)^2}{(1+ 2\alpha) (1+ 3\alpha)} \,(a^3 + b^3)$. Hence
 $$ \E[\Bar{X}^2] \leq \dfrac{2(1 +\alpha)^2}{(1+ 2\alpha) (1+ 3\alpha)} \,(a^3 + b^3) \leq \dfrac{2(1 +\alpha)^2}{(1+ 2\alpha) (1+ 3\alpha)} \,(a + b)^3 = \dfrac{2(1 +\alpha)^2}{(1+ 2\alpha) (1+ 3\alpha)}. $$
Translating back to an arbitrary mode $m$, we obtain
$$ M(\Bar{X})^2\, \E(\Bar{X}-m)^2 \leq \dfrac{2(1 +\alpha)^2}{(1+ 2\alpha) (1+ 3\alpha)}. $$
\end{proof}

The second-moment inequality \eqref{sub_optimial_bound} was obtained by transferring the corresponding continuous estimates to the discrete setting. In the following section, we use the majorization argument that exploits the discrete structure more directly and leads to sharper bounds.

\section{Sharper variance bounds}\label{max-var}

We now use the majorization argument from the preceding section more directly to the discrete setting to sharpen the bound relating the maximum probability and the variance of a discrete $\alpha$-concave random variable. This is achieved through sharper control of the second moment about the mode. The main idea is to compare a discrete $\alpha$-concave distribution with a suitably chosen two-sided $\alpha$-affine distribution. 

\begin{definition}

We say that a distribution is a two-sided discrete $\alpha$-affine distributions if its probability mass function is of the form
\begin{equation*}
  g(k) =
    \begin{cases}
      C(1 + \lambda_1 k)^{1/\alpha} & k \geq 0 \\
      C(1 - \lambda_2 k)^{1/\alpha} & k \leq 0
    \end{cases},
\end{equation*}
with $C, \lambda_1, \lambda_2 > 0$.
    
\end{definition}

Such distributions play the role of extremals for the majorization argument, as shown by the following proposition.

\begin{proposition}\label{sharp-Z}

Let $\alpha \in (-1/3 , 0)$. Let $X$ be a discrete $\alpha$-concave random variable on $\Z$ with mode at 0. Then, there exists a two-sided discrete $\alpha$-affine random variable $Z$ such that $M(Z) = M(X)$ and for any non-decreasing function $\Psi$,
$$ \Psi(\E[X^2]) \leq \Psi(\E[Z^2]). $$
    
\end{proposition}

\begin{proof}
Let us denote by $f$ the probability mass function of $X$. Construct $g$, the probability mass function of $Z$, such that
\begin{equation*}
  g(k) =
    \begin{cases}
      f(0)(1 + \lambda_1 k)^{1/\alpha} & k \geq 0 \\
      f(0)(1 - \lambda_2 k)^{1/\alpha} & k \leq 0
    \end{cases},
\end{equation*}
where $\lambda_1, \lambda_2 > 0$ are such that $\sum_{k \geq 0} g(k) = \sum_{k \geq 0} f(k)$ and $\sum_{k \leq 0} g(k) = \sum_{k \leq 0} f(k)$. By construction $g(0) = f(0)$. Thus, on $[0, +\infty)$, there exists a convex functions $V$ and an affine function $W$ such that $f = V^{1/\alpha}$ and $g = W^{1/\alpha}$ on $\N = \{0,1,2, \dots \}$. Since $V$ is convex and $W$ is affine, and $V(0) = W(0)$, the functions $V$ and $W$ cross at most once on $[0, \infty)$. A zero-crossing would mean that either $V \leq W$ or $V \geq W$ on the whole $[0, \infty)$, but that would contradicts that $f$ and $g$ have same total mass. Therefore, there is a one-crossing, that is there exists an interval $[0,x_0]$, for some $x_0>0$, such that $W \geq V$ on $[0, x_0]$ and $W \leq V$ on $[x_0, \infty)$. Equivalently, $g \leq f$ on $[0, x_0]$ and $g \geq f$ on $[x_0, \infty)$. But this necessarily implies that for all $t \geq 0$, 
$$ \P(Z \geq t) \geq \P(X \geq t). $$
As a consequence, by the layer-cake representation,
$$ \int_0^\infty x f(x) d\mu(x) \leq \int_0^\infty x g(x) d\mu(x), $$
where $\mu$ is the counting measure. Applying Theorem \ref{majorization} with $\mu$ as the counting measure yields 
$$ \int_0^\infty \varphi(x) f(x) d\mu(x) \leq \int_0^\infty \varphi(x) g(x) d\mu(x), $$
for any nondecreasing convex function $\varphi$. By taking $\varphi(x) = x^2$, which is nondecreasing on $[0, +\infty)$, we have
$$  \int_0^{\infty} x^2 f(x) d\mu(x) \leq \int_0^{\infty} x^2 g(x) d\mu(x), $$
that is
$$ \sum_{k \geq 0} k^2 f(k) \leq \sum_{k \geq 0} k^2 g(k). $$
By applying the same argument to $f(-k)$ and $g(-k)$ on $\N$ gives
$$ \sum_{k \geq 0} k^2 f(-k) \leq \sum_{k \geq 0} k^2 g(-k), $$
which yields $\E[X^2] \leq \E[Z^2]$, and the desired result follows.  
\end{proof}

Proposition \ref{sharp-Z} reduces the problem to obtaining suitable second-moment estimates for $\alpha$-affine distributions. Let us first treat the one-sided case on $\mathbb{N} = \{0,1,2, \dots\}$.

\begin{lemma}\label{bound-extremal}

Let $\alpha \in (-1/3 , 0)$. Let $Z$ be a discrete $\alpha$-affine random variable on $\N$ with mode at 0, then
    \begin{equation}\label{alpha-affine}
        M(Z)^2 \left( 1 + \frac{1}{2}c(\alpha) \mathbb{E}[Z^2] \right) \leq 1,
    \end{equation}
where $c(\alpha) = \dfrac{ ( 1+ 2\alpha)(1 + 3\alpha)}{(1+\alpha)^2}$.

\end{lemma}

\begin{proof}
Since the mode of $Z \sim g$ is at 0, one may write $g(k) = g(0) (1 + \lambda k)^{1/\alpha}$, with $\lambda > 0$. Denote $\beta = - \frac{1}{\alpha}$. Note that $g(0) = 1/F(\lambda)$, where 
$$ F(\lambda) = \displaystyle \sum_{k=0}^\infty (1+ \lambda k)^{-\beta}, $$
and $\mathbb{E}[Z^2] = G(\lambda)/F(\lambda)$, where 
$$ G(\lambda) = \displaystyle \sum_{k=1}^\infty k^2\,(1+ \lambda k)^{-\beta}. $$

\vskip3mm
\noindent
After plugging in $F$ and $G$, the equation \eqref{alpha-affine} becomes
\begin{equation}\label{to-prove}
c(-1/\beta) G \leq 2(F^3 - F), \qquad \text{for $\lambda>0$.}
\end{equation}
Denoting $S(\lambda) = F(\lambda) - 1 = \displaystyle\sum_{k=1}^\infty (1 + \lambda\,k)^{-\beta}$, proving \eqref{to-prove} is equivalent to proving
    \begin{equation}\label{to-prove-2}
        c(-1/\beta) G \leq 2S(S+1)(S+2), \qquad \text{for $\lambda>0$.}
    \end{equation}
Let $r = (\beta-1)\lambda$. We start by obtaining an upper bound for $G$ and a lower bound for $S$ in terms of the parameter $r$. Using Gamma integral representation, one can write
    \begin{equation}\label{Gamma-rep}
    (1+ \lambda k)^{-\beta} =\dfrac{1}{\Gamma(\beta)} \int_0^\infty x^{\beta -1}e^{-x} e^{-\lambda kx}\,dx.
    \end{equation}
    By Tonelli's theorem,
    $$ G(\lambda) = \dfrac{1}{\Gamma(\beta)}\sum_{k=1}^\infty k^2\,  \int_0^\infty x^{\beta -1}e^{-x} e^{-\lambda kx}\,dx = \dfrac{1}{\Gamma(\beta)} \int_0^\infty x^{\beta -1}e^{-x} \sum_{k=1}^\infty k^2\,e^{-\lambda kx}\,dx. $$
    Using hyperbolic sine and cosine identities
    $$ \displaystyle \sum_{k=1}^\infty k^2\,e^{-uk} = \dfrac{e^{-u}(1+e^{-u})}{(1- e^{-u})^3} = \dfrac{\cosh{(u/2)}}{4\sinh^3{(u/2)}}. $$
    Recall the Lazarevi\'c's inequality $(\sinh(x)/x)^p \geq \cosh (x)$, holding for $x>0$ and $p \geq 3$ (see, e.g.\cite{Mit}). Therefore,
    $$
    \sum_{k=1}^\infty k^2\,e^{-uk} \leq 2/u^3, $$
  which in turn yields
  \begin{align*}
  G(\lambda) \leq \dfrac{1}{\Gamma(\beta)} \int_0^\infty x^{\beta-1} \,e^{-x}\cdot \frac{2}{\lambda^3 x^3}\,dx 
  = \dfrac{2}{\lambda^3} \dfrac{1}{\Gamma(\beta)}\cdot \Gamma(\beta-3)
  = \dfrac{2}{\lambda^3 (\beta-1)(\beta-2)(\beta-3)}.
   \end{align*}
  Multiplying both sides by $c(-1/\beta)$,
  \begin{equation}\label{R}
    c(-1/\beta) G(\lambda) \leq \dfrac{2}{\lambda^3 (\beta - 1)^3} = \frac{2}{r^3}. 
    \end{equation}
   Similarly, we obtain a lower bound for $S$ using the Gamma representation and Tonelli's,
   \begin{equation*}
       S(\lambda) = \dfrac{1}{\Gamma(\beta)} \int_0^\infty x^{\beta-1} e^{-x}\, \frac{1}{e^{\lambda x}-1}\, dx.
   \end{equation*}
  Using the elementary inequality $\dfrac{1}{e^u -1} \geq \dfrac{1}{u} - \dfrac{1}{2}$, for $u>0$, we obtain
   \begin{align}\label{S}
       S(\lambda) \geq \frac{1}{\lambda}\cdot \dfrac{1}{\Gamma(\beta)}\int_0^\infty x^{\beta - 2} e^{-x}\,dx - \frac{1}{2} = \dfrac{1}{\lambda (\beta-1)} - \dfrac{1}{2} = \dfrac{1}{r} - \dfrac{1}{2}.
   \end{align}
    The above estimates obtained for $G$ and $S$ will be good enough to handle the case when $r$ is small. However, when $r$ is large, sharper estimates will be needed. For this reason, we treat two regimes in $r$.

    \vskip2mm
    \noindent
    {\bf {Case 1:}  $r \leq \dfrac{22}{13}$}
   \vskip2mm \noindent
   Using the lower bound \eqref{S} obtained for $S$,
   $$ 2S (S+1) (S+2) \geq 2 \left( \frac{1}{r} - \frac{1}{2} \right) \left(\frac{1}{r} + \frac{1}{2} \right) \left(\frac{1}{r} + \frac{3}{2}\right). $$
   Therefore, recalling \eqref{R}, to prove \eqref{to-prove-2} it is enough to show that
   $$ 2 \left( \frac{1}{r} - \frac{1}{2} \right) \left(\frac{1}{r} + \frac{1}{2} \right) \left(\frac{1}{r} + \frac{3}{2}\right) - \frac{2}{r^3} \geq 0, \qquad \text{for $r \leq 22/13$}. $$
   After simplification, we are left with the inequality $r^2 + \frac{2}{3} r - 4 \leq 0$, which can easily be verified to hold when $r \leq 22/13$.

   \vskip2mm
    \noindent
    {\bf {Case 2:}  $r \geq \dfrac{22}{13}$}
   \vskip2mm \noindent
In this regime, a different upper bound for $G$ is needed, which is derived using a standard sum-integral estimate. Let 
\begin{equation}\label{h}
h(x) = (x+1)^2 \left(1 + \frac{\lambda}{1+\lambda} x \right)^{-\beta},
\end{equation}
and note that
\begin{equation}\label{G-2}
\dfrac{G(\lambda)}{(1+\lambda)^{-\beta}} = \sum_{k=1}^\infty k^2 \dfrac{(1 + \lambda k )^{-\beta}}{(1+\lambda)^{-\beta}}  = \sum_{k=0}^\infty h(k).
\end{equation}
One can easily check that $h$ is a unimodal function on $[0, \infty)$, therefore 
\begin{equation}\label{sum-integral}
\sum_{k \geq 0} h(k) \leq \int_0^\infty h(x) dx + \sup_{x \geq 0} h(x).
\end{equation}
Indeed, denoting by $m$ the point where $h$ takes its maximum value, one has
\begin{eqnarray*}
\int_0^\infty h(x) dx & \geq & \sum_{k = 0}^{\lfloor m \rfloor - 1} \int_k^{k+1} h(x)dx + \int_{\lfloor m \rfloor}^{\lfloor m \rfloor+1} h(x)dx + \sum_{k \geq \lfloor m \rfloor + 1} \int_k^{k+1} h(x)dx \\ & = & \sum_{k \geq 0} h(k) - h(\lfloor m \rfloor) - h(\lfloor m \rfloor + 1) + \int_{\lfloor m \rfloor}^{\lfloor m \rfloor+1} h(x)dx,
\end{eqnarray*}
with the convention that the summation $\sum_{k = 0}^{\lfloor m \rfloor - 1}$ is 0 if $m < 1$. Since
$$ \int_{\lfloor m \rfloor}^{\lfloor m \rfloor + 1} h(x)dx \geq h(\lfloor m \rfloor)(m - \lfloor m \rfloor) + h(\lfloor m \rfloor + 1)(\lfloor m \rfloor + 1 - m), $$
one deduces that
$$ \sum_{k \geq 0} h(k) \leq \int_0^\infty h(x) dx + h(\lfloor m \rfloor)(1 - m + \lfloor m \rfloor) + h(\lfloor m \rfloor + 1)(m - \lfloor m \rfloor) \leq \int_0^\infty h(x) dx + \sup_{x \geq 0} h(x). $$
The integral of $h$ can be evaluated directly, and one has
$$ \int_0^\infty h(x) dx = \frac{2(1+\lambda)^3 + 2\lambda(1+\lambda)^2(\beta-3)+\lambda^2 (1+\lambda)(\beta-2)(\beta-3)}{\lambda^3(\beta-1)(\beta-2)(\beta-3)}, $$
so that
$$ c(-1/\beta) \int_0^\infty h(x) dx =  \frac{1}{\lambda^3 (\beta - 1)^3}(1+ \lambda) (\lambda^2 \beta^2 - 3 \lambda^2 \beta + 2 \lambda^2 + 2 \lambda \beta - 2\lambda + 2). $$
Recalling that $r = \lambda (\beta-1)$, the right-hand side can be rewritten as
$$ \frac{1}{r^3} \left(1+ \dfrac{r}{\beta - 1}\right) \left(r^2 (1 - \dfrac{1}{\beta - 1}) + 2r + 2\right).$$
Denoting $u = \dfrac{1}{\beta - 1} < \frac{1}{2}$, we arrive at
\begin{align}\label{bound-int-h}
c(-1/\beta) \int_0^\infty h(x)\,dx & = \dfrac{(1+ ru) (r^2 (1-u) + 2r + 2))}{r^3} \nonumber \\ & = \dfrac{1}{r} + \dfrac{2}{r^2} + \dfrac{2}{r^3} + u \left(\dfrac{1}{r} + \dfrac{2}{r^2}\right) + u(1-u) \nonumber \\ & \leq \dfrac{1}{4} + \dfrac{3}{2r} + \dfrac{3}{r^2} + \dfrac{2}{r^3}.
\end{align}
It remains to upper bound $c(-1/\beta) \sup_{x \geq 0} h(x)$. Note that $h'(x) = 0$ at the point $x_0 = \dfrac{2 - \beta \lambda^\prime}{(\beta-2) \lambda^\prime}$, where $\lambda^\prime = \dfrac{\lambda}{1+\lambda}$. If $2 - \beta \lambda^\prime \leq 0$, then $\sup_{x \geq 0} h(x)$ is attained at $x=0$, and thus
$$ c(-1/\beta) \sup_{x \geq 0} h(x) = c(-1/\beta) \cdot h(0) = c(-1/\beta) \cdot 1 \leq 1. $$
Otherwise, $2 - \beta \lambda^\prime \geq 0$ and a direct computation gives
\begin{align*}
\sup_{x\geq 0} h(x) = h(x_0) = \dfrac{4 (1- \lambda^\prime)^{\beta-2} (\beta-2)^{2 - \beta}}{\lambda'^2\, \beta^\beta} = \dfrac{4}{\lambda^2} \left(\dfrac{1+\lambda}{\beta}\right)^\beta (\beta-2)^{\beta-2},
\end{align*}
and thus, recalling $r = \lambda(\beta - 1)$,
\begin{align*}
c(-1/\beta) \sup_{x\geq 0} h(x)  & =  \dfrac{4}{\lambda^2 (\beta-1)^2} \left(\dfrac{1+\lambda}{\beta}\right)^\beta (\beta-2)^{\beta-1}\cdot (\beta-3)\\
& = \dfrac{4}{r^2} \cdot \dfrac{\beta-3}{\beta- 2} \cdot \left( \dfrac{(1+\lambda) (\beta-2)}{\beta} \right)^\beta.
\end{align*}
Observe that the condition $2 - \beta \lambda^\prime \geq 0$ implies $\dfrac{(1+\lambda) (\beta-2)}{\beta} \leq 1$. Hence, 
\begin{equation*}
c(-1/\beta) \sup_{x\geq 0} h(x) \leq \dfrac{4}{r^2}.
\end{equation*}
Finally, we obtain
\begin{equation}
c(-1/\beta) \sup_{x\geq 0} h(x) \leq  \max \left\{ 1, \dfrac{4}{r^2}\right\} = 
  \begin{cases}\label{bound-h}
            \dfrac{4}{r^2} & \text{if}\,\,\, \dfrac{17}{10} \leq r <2\\
      
      1  &  \text{if}\,\,\, r \geq 2
   \end{cases}.
\end{equation}
Recalling \eqref{G-2}, if $\dfrac{22}{13} \leq r < 2$, then combining \eqref{bound-int-h} and \eqref{bound-h},
\begin{align*}
    c(-1/\beta) G(\lambda) \leq (1+\lambda)^{-\beta} \left(\dfrac{1}{4} + \dfrac{3}{2r} + \dfrac{7}{r^2} + \dfrac{2}{r^3} \right) < 4 (1+\lambda)^{-\beta},
\end{align*}
which implies \eqref{to-prove-2} since $4 (1+\lambda)^{-\beta} \leq 4S \leq 2S(S+1)(S+2)$. If $r \geq 2$, then \eqref{bound-int-h} and \eqref{bound-h} gives
\begin{align*}
    c(-1/\beta) G \leq (1+\lambda)^{-\beta} \left(\dfrac{5}{4} + \dfrac{3}{2r} + \dfrac{3}{r^2} + \dfrac{2}{r^3}  \right) < 4 (1+\lambda)^{-\beta},
\end{align*}
which again implies \eqref{to-prove-2}. This completes the proof of \eqref{to-prove-2} in all regimes.
\end{proof}
\begin{remark}\label{rem-sharp}
    The bound \eqref{alpha-affine} is asymptotically sharp, with equality attained asymptotically within the class of one-sided $\alpha$-affine distributions. Indeed, letting $\lambda \to 0^{+}$, one can obtain the following estimates using Riemann-sum approximation.
\end{remark}

$$ \lim_{\lambda \to 0^+} \sum_{k=0}^\infty \lambda (1  + \lambda k)^{-\beta} = \int_0^\infty (1+x)^{-\beta}\,dx = \frac{1}{\beta-1}, $$
$$ \lim_{\lambda \to 0^+} \sum_{k=0}^\infty \lambda \, (\lambda k)^2 (1 + \lambda k)^{-\beta} =  \int_0^\infty x^2(1+x)^{-\beta} \,dx = \dfrac{2}{(\beta - 1)(\beta -2)(\beta -3)}. $$
Thus,
$$ F(\lambda) = \sum_{k=0}^\infty (1  + \lambda k)^{-\beta} \sim_{\lambda \to 0^+} \frac{1}{\lambda(\beta-1)}, $$
$$ G(\lambda) =\sum_{k=1}^\infty k^2 (1 + \lambda k)^{-\beta} \sim_{\lambda \to 0^+} \dfrac{2}{\lambda^3 (\beta - 1)(\beta -2)(\beta -3)}. $$
Therefore, as $\lambda \to 0^{+}$, $M(Z)= \dfrac{1}{F(\lambda)} \to 0$ and 
$$ M(Z)^2 \mathbb{E}[Z^2] = \frac{G(\lambda)}{F(\lambda)^3} \sim \dfrac{2(\beta-1)^2}{(\beta-2)(\beta-3)} = \dfrac{2}{c(-1/\beta)}. $$
Consequently,
$$M(Z)^2 \left( 1 + \frac{1}{2}c(\alpha) \mathbb{E}[Z^2] \right)  \to 1.$$
Thus, the family of one-sided $\alpha$-affine distributions with $\lambda\to0^+$ forms an asymptotically extremizing family for equation \eqref{alpha-affine}.

\iffalse
\begin{align*}
F(\lambda) =1+ \sum_{k=1}^\infty  (1  + \lambda k)^{-\beta}& = 1 + \dfrac{1}{\lambda} \sum_{k=1}^\infty  (1  + \lambda k)^{-\beta}\cdot \lambda\\
&\to 1+ \dfrac{1}{\lambda} \int_0^\infty (1+x)^{-\beta}\,dx \\
&= 1 + \dfrac{1}{\lambda} \cdot \dfrac{1}{\beta-1}.
\end{align*}
 \begin{align*}    
      G(\lambda) =\sum_{k=1}^\infty k^2 (1 + \lambda k)^{-\beta} &= \dfrac{1}{\lambda^3} \sum_{k=1}^\infty (\lambda k)^2 (1 + \lambda k)^{-\beta}\cdot \lambda\\ &\rightarrow \dfrac{1}{\lambda^3}\int_0^\infty x^2(1+x)^{-\beta}\,dx\\
    & = \dfrac{2}{\lambda^3 (\beta - 1)(\beta -2)(\beta -3)}.
       \end{align*}
       Thus $F(\lambda) \sim \frac{1}{\lambda(\beta-1)}$ and $G(\lambda) \sim \frac{2}{\lambda^3 (\beta - 1)(\beta -2)(\beta -3)}.$
\fi

\vskip2mm

Finally, let us extend \eqref{alpha-affine} to two-sided discrete $\alpha$-affine distributions on $\mathbb{Z}$.

\begin{lemma}\label{bound-extremal-V}

Let $\alpha \in (-1/3 , 0)$. Let $Z$ be a two-sided discrete $\alpha$-affine distributions with mode at 0, then
    \begin{equation}\label{equation-V}
        M(Z)^2 \left( 1 + \frac{1}{2}c(\alpha) \mathbb{E}[Z^2] \right) \leq 1,
    \end{equation}
where $c(\alpha) = \dfrac{ ( 1+ 2\alpha)(1 + 3\alpha)}{(1+\alpha)^2}$.

\end{lemma}

\begin{proof}
Let $Z \sim g$ be a random variable with a two-sided discrete $\alpha$-affine distributions, so that
\begin{equation*}
  g(k) =
    \begin{cases}
      g(0)(1 + \lambda_1 k) & k \geq 0 \\
      g(0)(1 - \lambda_2 k) & k \leq 0
    \end{cases},
\end{equation*}
with fixed $\lambda_1,\lambda_2 > 0$. Let us recall the following notations from the proof of Lemma \ref{bound-extremal},
$$ S(\lambda) = \sum_{k \geq 1} (1+\lambda k)^{1/\alpha}, \qquad G(\lambda) = \sum_{k \geq 1} k^2 (1+\lambda k)^{1/\alpha}. $$
Note that
$$ g(0) = \frac{1}{1 + S(\lambda_1) + S(\lambda_2)}, $$
and
$$ \E[Z^2] = \frac{G(\lambda_1) + G(\lambda_2)}{1 + S(\lambda_1) + S(\lambda_2)}. $$
Therefore, inequality \eqref{equation-V} can be restated as
\begin{equation}\label{to-prove-V}
c(\alpha)(G(\lambda_1) + G(\lambda_2)) \leq 2 \left[(1 + S(\lambda_1) + S(\lambda_2))^3 - (1 + S(\lambda_1) + S(\lambda_2)) \right].
\end{equation}
Recall that we have proven inequality \eqref{to-prove-2}, which states that
$$ c(\alpha) G(\lambda) \leq 2 \left[(1+S(\lambda))^3 - (1+S(\lambda)) \right], \qquad \lambda=\lambda_1, \lambda_2. $$
Performing summation yields,
\begin{eqnarray*}
c(\alpha) (G(\lambda_1) + G(\lambda_2)) & \leq & 2 \left[(1+S(\lambda_1))^3 + (1+S(\lambda_2))^3 - (1+S(\lambda_1)) - (1+S(\lambda_2)) \right] \\ & \leq & 2 \left[(1+S(\lambda_1)+S(\lambda_2))^3 - (1+S(\lambda_1)+S(\lambda_2)) \right],
\end{eqnarray*}
where the last inequality follows by convexity of $t \mapsto t^3 - t$ ($t \geq 0$). Therefore we have established \eqref{to-prove-V}.
\end{proof}

We are now ready to prove Theorem \ref{maximum-vs-variance}.

\begin{proof}[Proof of Theorem \ref{maximum-vs-variance}]
By translation, we may assume that $X$ has a mode at $0$. Let $Z$ be a random variable with a two-sided $\alpha$-affine distributions such that $M(Z) = M(X)$. By Proposition \ref{sharp-Z} applied to the function $\Psi(x) = \sqrt{1 + (c(\alpha)/2) \, x}$, we obtain
$$ M(X)^2 \left( 1 + \frac{1}{2}c(\alpha) \mathbb{E}[X^2] \right) \leq M(Z)^2 \left( 1 + \frac{1}{2}c(\alpha) \mathbb{E}[Z^2] \right) \leq 1, $$
where the last inequality comes from Lemma \ref{bound-extremal-V}. The desired result follows from $\Var(X) \leq \E[X^2]$.
\end{proof}
\begin{remark}\label{sharpness}
    Contrary to the second-moment bound in Lemma \ref{bound-extremal}, the constant in the variance bound of Theorem \ref{maximum-vs-variance} is not optimal and can be improved. Indeed, for the same one-sided $\alpha$-affine family considered in Remark \ref{rem-sharp}, the same Riemann-sum argument gives, as $\lambda \to 0^+$,
    \end{remark}\noindent
$$ M(Z) \sim \lambda (\beta-1), \quad \mathbb{E}[Z] \sim \dfrac{1}{\lambda(\beta-2)}, \quad \mathbb{E}[Z^2] \sim \dfrac{2}{\lambda^2\,(\beta-2)(\beta-3)}. $$
Consequently
$$M(Z)^2 \Var(Z) \to  \dfrac{(\beta-1)^3}{(\beta-2)^2\, (\beta-3)}.$$
Therefore, if 
$$M(X)^2 \left( 1 + c^\prime(\alpha) \Var(X) \right) \leq 1$$
holds for any discrete $\alpha$-concave random variable $X$, then necessarily
$$c^\prime(\alpha) \leq \dfrac{(\beta-2)^2 (\beta-3)}{(\beta-1)^3} = \dfrac{(1+2\alpha)^2(1+3\alpha)}{(1+\alpha)^3},$$
with equality attained asymptotically for one-sided $\alpha$-affine distributions, suggesting the following conjecture.\vskip2mm

    \begin{conj} \label{conj_maximum-vs-variance}
        Let $\alpha \in (-1/3 , 0)$. If $X$ is a discrete $\alpha$-concave random variable on $\Z$, then
    \begin{equation*}
    M(X)^2 \left( 1 + k(\alpha)\Var(X) \right) \leq 1,\end{equation*}
where $k(\alpha) = \dfrac{ (1+ 2\alpha)^2(1+ 3\alpha)}{(1+\alpha)^3}$.
    \end{conj}
 
Note that, in the limit $\alpha \to 0^-$, Conjecture \ref{conj_maximum-vs-variance} recovers the corresponding entropy-variance inequalities in the log-concave setting; see \cite{Ara, JMNS, GMPP}.

\section{Concentration function}\label{Levy-concentration}

This section discusses the L\'evy concentration function of a discrete $\alpha$-concave distribution. Given a random variable $X$, its L\'evy concentration is defined as
$$ M(X ;\ell) = \sup_{x \in \R} \P(x \leq X \leq x + \ell), \quad \ell \geq 0. $$
We refer to \cite{MR, P, BC15} for further background on the concentration function. Note that $M(X;0)=M(X)$. It is known that for any discrete random variable $X$ and any integer $\ell \geq 0$ (see \cite[Proposition 8.3]{BMM}),
$$ M(X;\ell) \, \geq \, \frac{\ell + 1}{\sqrt{1 + \ell(\ell + 2) + 12\,\Var(X)}}. $$
The following upper bound can be obtained under $\alpha$-concavity.

\begin{proposition}

Let $\alpha \in (-1/3 , 0)$. Let $X$ be a discrete $\alpha$-concave random variable. Then, for any integer $\ell \geq 0$,
$$ M(X;\ell) \leq \frac{\ell + 1}{\sqrt{1 + \frac{\ell(\ell+2)}{2} + \frac{c(\alpha)}{4} \Var(X)}}, $$
where $c(\alpha) = \dfrac{ ( 1+ 2\alpha)(1 + 3\alpha)}{(1+\alpha)^2}$.

\end{proposition}

\begin{proof}
Applying \cite[Lemma 8.1]{BMM}, one has for $\ell \in \N$
\begin{equation}\label{Levy-function}
M(X;\ell) \, = \, (\ell + 1)\, M(X+U_\ell),
\end{equation}
where the random variable $U_\ell$ is independent of $X$ and has a discrete uniform distribution on the integer interval $\{0,1,\dots,\ell\}$. Contrary to the argument used in the log-concave setting in \cite[Section 8]{BMM}, the convolution of $U_{\ell}$ with a discrete $\alpha$-concave random variable is no longer $\alpha$-concave (when $\alpha < 0$), so one may not apply Theorem \ref{maximum-vs-variance} to $X+U_{\ell}$. Instead, we will make use of the inequality
\begin{equation}\label{trivial}
M(X+U_{\ell}) \leq \left[ \frac{M(X)^{-2} + M(U_{\ell})^{-2}}{2} \right]^{-\frac{1}{2}},
\end{equation}
which follows from the trivial bound $M(X+U_{\ell}) \leq \min(M(X), M(U_{\ell}))$ together with the fact that $\min(a,b) \leq \left(\frac{a^{-2}+b^{-2}}{2}\right)^{-1/2}$ for $a,b \geq 0$. Since $M(U_{\ell})^{-2} = (1+\ell)^2$, Theorem \ref{maximum-vs-variance} together with \eqref{trivial} leads to
$$ M(X+U_{\ell}) \leq \left[ \frac{1}{2} \left( 1 + \frac{c(\alpha)}{2} \Var(X) + 1 + \ell(\ell+2) \right) \right]^{-1/2}, $$
which yields the desired result using \eqref{Levy-function}.
\end{proof}

\section{Reverse Entropy Power Inequality}\label{Reverse-EPI}

Entropy power inequalities usually do not admit a reverse form without additional structural assumptions. Under suitable convexity assumptions, however, reverse entropy power inequalities are known to hold in both the continuous and discrete settings; see, for example, \cite{BM, BNT, BN, BMM, XMM}. 

Our argument follows the approach of \cite{BMM}, where the corresponding result for discrete log-concave random variables is obtained through a comparison between entropy and variance. Here, we combine the same general strategy with the maximum-variance inequality established in Theorem \ref{maximum-vs-variance}. We start by recalling the following proposition from \cite{BMM}.

\begin{proposition}{\cite[Proposition 3.3]{BMM}}\label{BMM_Proposition}
    Let $1<r \leq \infty$. If $X$ is an integer-valued random variable with finite variance, then
    $$1 \leq N_r(X) \leq 1+\dfrac{4(3r-1)}{r-1}\Var(X).$$
\end{proposition}

\begin{proof}[Proof of Theorem \ref{Reverse_EPI}]
    By Proposition \ref{BMM_Proposition}, one has for arbitrary integer-valued random variable $X$,
    $$ \Delta_r(X) \leq \frac{4(3r - 1)}{r-1} \Var(X). $$
    On the other hand, by Theorem \ref{maximum-vs-variance}, if $X$ is $\alpha$-concave,
    $$ \frac{c(\alpha)}{2} \Var(X) \leq \Delta_{\infty}(X) \leq \Delta_r(X), $$
    where the last inequality comes from monotonicity of $r \mapsto \Delta_r(X)$. Therefore,
    \begin{eqnarray*}
    \Delta_r \left( \sum_{i=1}^n X_i \right) \leq \frac{4(3r - 1)}{r-1} \Var \left( \sum_{i=1}^n X_i \right) & = & \frac{4(3r - 1)}{r-1} \sum_{i=1}^n \Var \left( X_i \right) \\ & \leq & \frac{4(3r - 1)}{r-1} \frac{2}{c(\alpha)} \sum_{i=1}^n \Delta_r(X_i).
    \end{eqnarray*}
\end{proof}

\section{Concentration inequalities}\label{concentration}

In this section, we derive concentration inequalities for discrete $\alpha$-concave random variables. Our argument is based on the single-crossing comparison with discrete $\alpha$-affine distributions used in the proof of Proposition \ref{sharp-Z}. In particular, the same comparison that yields the moment estimates of section \ref{max-var} also provides a stochastic domination of the tails by those of suitable one-sided $\alpha$-affine distributions. 

First we need the following estimate for the tails of a discrete $\alpha$-affine distributions.
\begin{lemma}\label{tec_lemma}
    Let $\beta>1$ and $\lambda>0$. Then for all $t > 0$,
    
    \begin{equation} \label{tail_bound}
        \dfrac{\displaystyle \sum_{k= \lceil t \rceil}^\infty (1+ \lambda k)^{-\beta}}{\displaystyle \sum_{k=0}^\infty (1+ \lambda k)^{-\beta}} \leq (1+ \lambda \lceil t \rceil)^{1-\beta}.
    \end{equation}
\end{lemma}
\begin{proof}
Note that
\begin{equation}\label{F-b}
    \sum_{k=\lceil t \rceil}^\infty (1+ \lambda k)^{-\beta} = \sum_{k=0}^\infty (1+ \lambda \lceil t \rceil + \lambda k)^{-\beta} = (1+ \lambda \lceil t \rceil )^{-\beta} F_{\beta}\left(\frac{\lambda}{1+\lambda \lceil t \rceil}\right),
\end{equation}
where $F_\beta (x) = \displaystyle \sum_{k=0}^\infty (1+xk)^{-\beta}$. We claim that $\lambda \mapsto \lambda F_\beta (\lambda)$ is non-decreasing. Indeed, by the Gamma integral representation  (see \eqref{Gamma-rep}) and Tonelli's theorem, write
$$ \lambda F_\beta (\lambda) = \dfrac{1}{\Gamma(\beta)} \int_0^\infty x^{\beta-1}e^{-x}\frac{\lambda}{1-e^{-\lambda x}}\,dx. $$
For every $x\geq 0$,
$$ \dfrac{d}{d\lambda} \left[ \dfrac{\lambda}{1-e^{-\lambda x}}\right] = \dfrac{1-e^{-\lambda x}(1+ \lambda x)}{(1-e^{-\lambda x})^2} \geq 0, $$
since $e^{\lambda x} \geq 1+ \lambda x$. Therefore $\dfrac{d}{d\lambda}(\lambda F_\beta(\lambda))\geq 0$. Consequently, we have
$$ \dfrac{\lambda}{1+\lambda \lceil t \rceil}F_\beta \left(\frac{\lambda}{1+\lambda \lceil t \rceil}\right) \leq \lambda F_\beta(\lambda), $$
which is the same as
$$ F_\beta \left(\frac{\lambda}{1+\lambda \lceil t \rceil}\right)  \leq (1+\lambda \lceil t \rceil)\, \sum_{k=0}^\infty (1+ \lambda k)^{-\beta}. $$
Substituting this into \eqref{F-b} proves \eqref{tail_bound}.
\end{proof}

Let us now prove Theorem \ref{concentration_bound}.

\begin{proof}[Proof of Theorem \ref{concentration_bound}]
By a translation argument, let us assume that $m=0$. Let 
$$ q_+ = \mathbb{P}(X \geq 0) \quad \text{and} \quad q_- = \mathbb{P}(X \leq 0). $$
By the construction used in the proof of proposition \ref{sharp-Z}, there exists $\lambda_{+}>0$ such that the one-sided $\alpha$-affine sequence
$$ g_+(k) = M(X)(1+ \lambda_{+}\,k)^{-\beta}, \quad k=0,1,2,\dots $$ satisfies
$$ \sum_{k=0}^\infty g_+(k) = \sum_{k=0}^\infty f(k) = q_{+}. $$
Thus 
\begin{equation}\label{q+}
q_+ = M(X)F_\beta(\lambda_+).
\end{equation}
Moreover, the single-crossing argument used in the proof of Proposition \ref{sharp-Z} shows that $f-g_+$ is non-negative before the crossing and non-positive afterward. Since $f$ and $g_+$ have the same mass on $\mathbb{N}$, it follows that
$$ \sum_{k=n}^\infty f(k) \leq \sum_{k=n}^\infty g_+(k), \quad \text{for all $n \in \mathbb{N}$.} $$
Thus for all $t>0$,
\begin{equation}\label{large-dev}
\mathbb{P}(X \geq t) = \mathbb{P}(X \geq \lceil t \rceil) \leq M(X) \sum_{k=\lceil t \rceil}^\infty (1+ \lambda_{+}\,k)^{-\beta}.
\end{equation}
Combining \eqref{q+} and \eqref{large-dev} together with Lemma \ref{tec_lemma} yields,
\begin{equation}\label{large-dev2}
\mathbb{P}(X \geq t) \leq q_{+}\,(1+ \lambda_{+}\lceil t \rceil)^{1-\beta}.\end{equation}
It remains to eliminate $\lambda_{+}$. Since $x \mapsto (1+\lambda_+ x)^{-\beta}$ is nonincreasing,
$$F_\beta(\lambda_+) = \sum_{k=0}^\infty (1+\lambda_{+}k)^{-\beta} \geq \int_0^\infty (1+\lambda_{+}x)^{-\beta}\,dx = \dfrac{1}{\lambda_+(\beta-1)}. $$
Together with \eqref{q+}, this yields
\begin{equation}\label{l+}
\lambda_{+} \geq \dfrac{M(X)}{q_+ (\beta-1)}.
\end{equation}
Since $1-\beta<0$ and $\lceil t \rceil \geq t$, equations \eqref{large-dev2}-\eqref{l+} give
\begin{equation}\label{large-dev3}
\mathbb{P}(X \geq t) \leq  q_{+}\left(1+ \dfrac{M(X)\,t}{q_+ (\beta-1)}\right)^{1-\beta}.  \end{equation}
Applying the above positive half-line argument to the reflected random variable $Y=-X$ , which is again discrete $\alpha$-concave with mode at 0, yields
\begin{equation}\label{large-dev4}
\mathbb{P}(Y \geq t)= \mathbb{P}(X \leq -t) \leq q_{-}\,\left(1+ \dfrac{M(X)\,t}{q_- (\beta-1)}\right)^{1-\beta}.\end{equation}
Combining \eqref{large-dev3} and \eqref{large-dev4}, we have for all $t>0$
\begin{equation}\label{large-dev5}
    \mathbb{P}(|X| \geq t) \leq  q_{+}\left(1+ \dfrac{M(X)\,t}{q_+ (\beta-1)}\right)^{1-\beta} + q_{-}\left(1+ \dfrac{M(X)\,t}{q_- (\beta-1)}\right)^{1-\beta}.
\end{equation}
It remains to maximize \eqref{large-dev5} over $q_{-}$ and $q_{+}$. To this end, note that
$$ q_{+} + q_{-} = 1 + M(X) $$
and $M(X) \leq q_{\pm}\leq 1$. Put $r = \beta -1$ and define
$$ \Phi(q) = q\left(1+ \frac{M(X)\,t}{qr}\right)^{-r}, $$
so that the right-hand side of \eqref{large-dev5} can be rewritten as $\Phi(q_+) + \Phi(q_-)$. Differentiating the function $\Phi$ twice with respect to $q$ yields
$$ \Phi''(q) = \dfrac{r(r+1)\left(\frac{M(X)\,t}{r}\right)^2\, q^{r-1}}{\left(q+ \frac{M(X)\,t}{r}\right)^{r+2}}, $$
which is nonnegative. Therefore, $\Phi$ is convex. As a result, the function
$$ H(q) = \Phi(q) + \Phi(1+M(X)-q) $$
is convex on $[M(X),1]$, and thus the maximum of $H$ occurs at an end-point. But since $H(M(X)) = H(1)$, we have
$$ H(q)\leq H(1) = \left(1+ \dfrac{M(X)\,t}{\beta-1}\right)^{1-\beta} + M(X)\left(1+ \dfrac{t}{\beta-1}\right)^{1-\beta}, $$
which proves the main result.
\end{proof}

\section*{Acknowledgments}

The first author would like to extend his appreciation to the Deanship of Scientific Research at King Saud University for funding this work through Waed Program (W25-76).

\section*{Tool and computational resource disclosure}

The authors used OpenAi ChatGPT GPT-5.6 Luna model as an auxiliary tool in deriving certain estimates used in the proof of Lemma \ref{bound-extremal}.


\begin{thebibliography}{50}


\bibitem{Ara} H. Aravinda. Entropy-variance inequalities for discrete log-concave random variables via degree of freedom. Discrete Mathematics 347.1 (2024): 113683.

\bibitem{BNT} K. Ball, P. Nayar, T. Tkocz. A reverse entropy power inequality for log-concave random vectors. Studia Math. 235 (2016), no. 1, 17--30.

\bibitem{BN} M. Bia\l obrzeski, P. Nayar. R\'enyi entropy and variance comparison for symmetric log-concave random variables. IEEE Transactions on Information Theory 69.6 (2023): 3431--3438.

\bibitem{Bobkov} S. G. Bobkov. Large deviations and isoperimetry over convex probability measures with heavy tails. Electron. J. Probab. (2007), 1072--1100.

\bibitem{BC15} S. G. Bobkov, G. P. Chistyakov. On concentration functions of random variables. J. Theor.
Probab. 28 (2015), no. 3, 976--988.

\bibitem{BM} S. Bobkov, and M. Madiman, Reverse Brunn-Minkowski and reverse entropy power inequalities for convex measures. Journal of Functional Analysis, 262.7 (2012): 3309-3339.

\bibitem{BMM} S. G. Bobkov, A. Marsiglietti, J. Melbourne. Concentration functions and entropy bounds for discrete log-concave distributions. Combin. Probab. Comput. 31 (2022), no. 1, 54--72.

\bibitem{B1} C. Borell. Convex measures on locally convex spaces, Ark. Math. 12 (1974), 239--252.

\bibitem{B2} C. Borell. Convex set functions in d-space. Period. Math. Hungar., 6 (1975), 111--136. 

\bibitem{BLM} S. Boucheron, G. Lugosi, and P. Massart. Concentration Inequalities: A Nonasymptotic
Theory of Independence. Oxford University Press, 2013.

\bibitem{FGP} M. Fradelizi, O. Guédon and A. Pajor. Thin-shell concentration for convex measures. Studia Mathematica (2014), 123--148.

\bibitem{GMPP} T. Gaxiola, J. Melbourne, V. Pigno, E. Pollard. Anti-concentration inequalities for log-concave variables on the real line. Preprint. arXiv:2505.05793.

\bibitem{HAT} S. Haghighatshoar, E. Abbe, I. E. Telatar. A new entropy power inequality for integer-valued random variables. IEEE Trans. Inform. Theory 60 (2014), no. 7, 3787--3796.

\bibitem{HV} P. Harremo\"es, V. Vignat. An entropy power inequality for the binomial family. J. Inequal. Pure
Appl. Math., vol. 4, no. 5, pp. 1--15, 2003.

\bibitem{JMNS}  J. Jakimiuk, D. Murawski, P. Nayar, S. S\l obodianiuk. Log-concavity and discrete degrees of freedom. Discrete Math. 347(6), 114020 (2024).

\bibitem{JY} O. Johnson, Y. Yu. Monotonicity, thinning, and discrete versions of the entropy power inequality. IEEE Trans. Inform. Theory 56 (2010), no. 11, 5387--5395.

\bibitem{KN} S. Karlin and A. Novikoff. Generalized convex inequalities. Pacific J. Math., 13(4):1251--1279, 1963.

\bibitem{K} E. Kowalski. An introduction to probabilistic number theory. Cambridge Studies in Advanced Mathematics, Series Number 192, 2021.

\bibitem{L} W. Li. Random texts exhibit Zipf's-law-like word frequency distribution. IEEE Transactions on information theory 38 (1992), no. 6, 1842--1845.

\bibitem{LH} G. D. Lin, C-Y. Hu. The Riemann zeta distribution. Bernoulli 7 (2001), no. 5, 817--828.

\bibitem{MMR} M. Madiman, J. Melbourne, C. Roberto. Bernoulli sums and R\'enyi entropy inequalities. Bernoulli 29 (2), 1578--1599.

\bibitem{MM22} A. Marsiglietti, J. Melbourne. Concentration inequalities for log-concave sequences. International Mathematics Research Notices, Volume 2026, Issue 4, February 2026, rnag023.

\bibitem{MM26} A. Marsiglietti, J. Melbourne. A note on the Littlewood-Offord problem for discrete log-concave distributions. Preprint, arXiv:2510.25869 [math.PR].

\bibitem{MR} A. L. Miroshnikov, B. A. Rogozin. Inequalities for concentration functions. Teor. Veroyatnost. i Primenen (Russian) 25(1), 178–183 (1980).

\bibitem{Mit} D. S. Mitrinovi\'c. Analytic Inequalities. Springer, New York, NY, USA, 1970.

\bibitem{N} T. Nakamura. A complete Riemann zeta distribution and the Riemann hypothesis. Bernoulli 21(1): 604--617, 2015.

\bibitem{PPS} A. Pachon, F. Polito, L. Sacerdote. Random Graphs Associated to Some Discrete and Continuous Time Preferential Attachment Models. Journal of Statistical Physics, 162(6): 1608--1638, 2016.

\bibitem{P} V. V. Petrov. Limit theorems of probability theory. Sequences of independent random variables. Oxford
Studies in Probability, 4. Oxford Science Publications. The Clarendon Press, Oxford University Press, New York (1995) xii+292 pp.

\bibitem{S} H. A. Simon. On a class of skew distribution functions. Biometrika (1955), Volume 42, Issue 3-4, 425--440.

\bibitem{T} G. Tenenbaum. Introduction to analytic and probabilistic number theory, Cambridge Studies in Advanced Mathematics, Vol. 46, Cambridge University Press, 1995.

\bibitem{V} R. Vershynin. High-Dimensional Probability: An Introduction with Applications in Data Science. 2nd ed. Cambridge University Press; 2026.

\bibitem{W} W. Whitt. Uniform conditional variability ordering of probability distributions. Journal of Applied Probability, 22(3):619--633, 1985.

\bibitem{WM} J. O. Woo, M. Madiman. A discrete entropy power inequality for uniform distributions. Proc. IEEE Intl. Symp. Inform. Theory, Hong Kong, China, 2015.

\bibitem{XMM} P. Xu, J. Melbourne, M. Madiman. Reverse entropy power inequalities for s-concave densities. IEEE International Symposium on Information Theory (ISIT). IEEE, 2016.


\end{thebibliography}
\end{document}